\documentclass[12pt,a4paper]{article}
\usepackage{amsmath,amsfonts,amsthm,amssymb,latexsym}
\usepackage{xspace}
\usepackage{xcolor}
\usepackage{enumitem}
\usepackage{graphicx}
\date{27th May 2024}

\newcommand{\Aut}{\mathrm{Aut}}
\newcommand{\Out}{\mathrm{Out}}
\newcommand{\Sym}{\mathrm{Sym}}
\newcommand{\Alt}{\mathrm{Alt}}
\newcommand{\PSL}{\mathrm{PSL}}
\newcommand{\PGamL}{\mathrm{P \Gamma L}}

\newcommand{\Fun}{\mathrm{Fun}}

\newtheorem{theorem}{Theorem}[section] 
\newtheorem{lemma}[theorem]{Lemma} 

\newtheorem{corollary}[theorem]{Corollary}

\title{Quotients of permutation groups by nonabelian minimal normal subgroups}
\author{Derek F. Holt}
\begin{document}
\maketitle
\begin{abstract}
We prove that a quotient $G/N$ of a subgroup $G$ of $\Sym(n)$ by a nonabelian
minimal normal subgroup $N$ of $G$ embeds into $\Sym(m)$ for some $m<n$.
This result was proved previously in \cite{Chamberlain}, and we also prove
that, if $G$ is transitive then we can take $m \le 2n/5$.
\end{abstract}

\section{Introduction}
For a finite group $G$ we define $P(G)$ to be the smallest degree of a
faithful permutation representation of $G$; that is the smallest $n$ for which
there is an embedding $G \to \Sym(n)$. It was shown by Neumann in~\cite{Neu}
that $P(G/N)$ can be exponentially larger than $P(G)$ for normal
subgroups $N$ of $G$. In~\cite{EP}, Easdown and Praeger called
a quotient \emph{exceptional} if $P(G/N) > P(G)$ and proved that
various type of quotients of finite groups $G$, such as $G/O_p(G)$ for
a prime $p$, and $G/S(G)$ with $S(G)$ the solvable radical of $G$,
are not exceptional.

It was proved in \cite{Chamberlain} that $G/N$ is not exceptional when $N$ is
a nonabelian minimal normal subgroup of $G$, and we prove a stronger result for
transitive subgroups $G$ of $\Sym(n)$.

\begin{theorem}\label{thm:main}
Let $G \le \Sym(n)$, and let $N$ be a nonabelian minimal normal subgroup of $G$.
Then $G/N$ embeds into $\Sym(m)$ for some $m<n$.
Furthermore, if $G$ is a transitive subgroup of $\Sym(n)$, then
we can choose $m \le 2n/5$.
\end{theorem}

The examples $G=\Sym(5)$, $N=\Alt(5)$  with $n=5$, and $G=\PGamL(2,9)$,
$N=\PSL(2,9)$ with $n=10$, together with
wreath products of these groups by various other transitive groups, show that
the bound $m \le 2n/5$ for transitive subgroups of $\Sym(n)$ is best possible.

By repeated applications of this result, we get the immediate corollary:
\begin{corollary}\label{cor:main}
Let $G \le \Sym(n)$, and let $N$ be a nontrivial normal subgroup of $G$ with
no abelian composition factors. Then $G/N$ embeds into $\Sym(m)$ for some $m<n$.
Furthermore, if $G$ is a transitive subgroup of $\Sym(n)$, then
we can choose $m \le 2n/5$.
\end{corollary}

One motivation for investigating this particular problem was the algorithm
described in~\cite{LT} for finding a smallest sized generating set of a finite
group $G$, and the probabilistic variation and Magma implementation described
in~\cite{HT}. That algorithm starts by computing a chief series of a finite
group $G$. The chief series computed by the Magma intrinsic has the
chief factors in $S(G)$ at the bottom of the series, but we observed that
the algorithm often runs faster when using a chief series with as many
nonabelian chief factors at the bottom as possible. Such a chief series can be
found more easily in a permutation group if we have an efficient method
of representing quotients $G/N$ for which the composition factors of $N$ are
nonabelian.

\section{Preliminary lemmas}
We assume throughout the paper that $G$ is a finite group,
although the first two lemmas are valid for arbitrary groups.

\begin{lemma}\label{lem:ntproj}
Let $N \unlhd G$ with $N$ a direct product of nonabelian
simple groups,
and let $K$ be a normal subgroup of $G$ with $K \cap N = 1$. Let $L$ be a normal
subgroup of $G$ with $L \le KN \cong K \times N$. Then $L$ contains all simple
direct factors of $N$ onto which elements of $L$ project non-trivially.
\end{lemma}
\begin{proof} Identify $KN$ with $K \times N$, and suppose that
there exists $(k,n) \in L$ such that $n$ has non-trivial projection $n_s$ onto
a simple direct factor $S$ of $N$. Then, since $Z(S) = 1$, there exists
$s \in S$ with $s^{-1}n_s s \ne n_s$ and then the commutator $[(k,n),(1,s)]
= (1,[n_s,s]) \in L \cap S$, so $L \cap S \ne 1$ and the simplicity of $S$
implies that $S \le L$.
\end{proof}

\begin{lemma}\label{lem:normsd}
Let $N = S^k$ for a nonabelian simple group $S$, and let
$H \le N$ be a subdirect product (i.e. $H$ projects onto each of the
$k$ simple factors of $N$). Then $N_N(H)=H$.
\end{lemma}
\begin{proof} We use induction on $k$. When $k=1$ we have $H=N$ and the result
holds. Suppose $k>1$, write $N = S \times T$ with $T = S^{k-1}$, and let
$g \in N_N(H)$. Then by the inductive hypothesis, the projection of $N_N(H)$
onto $T$ lies in the projection of $H$ onto $T$ so, by replacing $g$ by a
product with an element in $H$, we can assume that $g \in S$. If $g \ne 1$,
then there exists $s \in S$ with $[g,s] \ne 1$ and, since $H$ projects onto $S$,
 we have $[g,s] \in H$. But then, since $S \cap H \unlhd S$,
we get $S \le H$ and so $g \in H$, which proves the result.
\end{proof}

\begin{lemma}\label{lem:subembed}
Let $H \le G$ with $|G:H|=k$ and suppose that there is an embedding
$H \to \Sym(n)$ for some $n$. Then there is an embedding $G \to \Sym(kn)$.
\end{lemma}
\begin{proof}
Suppose that the image of the embedding $H \to \Sym(n)$ has $t$ orbits with
point stabilizers $H_1,\ldots,H_t$. Then we can define an embedding
$G \to \Sym(kn)$ such that the image also has $t$ orbits, with the same
point stabilizers $H_1,\ldots,H_t$.
\end{proof}

The wreath product $W := U \wr_\Gamma V$ of groups
$U \le \Sym(\Delta)$ and $V \le \Sym(\Gamma)$ with product action is defined
in~\cite[Section 2.1]{DM}. This (right) action is on the set
$\Omega := \Fun(\Gamma,\Delta)$ with $W = \Fun(\Gamma,U) \rtimes V$ and
$$\phi^{(f,v)}(\gamma) := \phi(\gamma^{v^{-1}})^{f(\gamma^{v^{-1}})}
\ \ \ \mathrm{for}\ \phi \in \Omega,\,f \in \Fun(\Gamma,U),\,v \in V,\,
\gamma \in \Gamma.
$$
In particular, taking $f=1$, the identity element of $\Fun(\Gamma,U)$ that
maps all elements of $\Gamma$ to the identity element of $U$, and
identifying $V$ with the subgroup $\{(1,v): v \in V\}$, the action of $V$
on $\Omega$ is given by $\phi^v(\gamma) = \phi(\gamma^{v^{-1}})$.

\begin{lemma}\label{lem:prodactorb}
Let $W := U \wr_\Gamma V$ be the wreath product with product
action on $\Omega$ as just defined, with $|\Delta|>1$, and let
$\alpha,\beta \in \Delta$ with $\alpha \ne \beta$. Then there is subset
$\Gamma_{\alpha\beta}$ of $\Omega$ that is fixed by the subgroup $V$ of $W$,
and for which the action of $V$ on $\Gamma_{\alpha\beta}$ is equivalent to its
action on $\Gamma$.
\end{lemma}
\begin{proof} For $\gamma \in \Gamma$ define $\phi_\gamma \in \Omega$
by $\phi_\gamma(\gamma)=\alpha$ and $ \phi_\gamma(\delta)=\beta$ for
$\delta \in \Gamma \setminus \{\gamma\}$, and let
$\Gamma_{\alpha\beta} := \{ \phi_\gamma : \gamma \in \Gamma \}$.
It is straightforward to check that $\Gamma_{\alpha\beta}$ has the required
property with the equivalence $\gamma \mapsto \phi_\gamma$ between
$\Gamma$ and $\Gamma_{\alpha\beta}$.
\end{proof}

The O'Nan-Scott Theorem for a finite primitive permutation group $G$ of degree
$n$ on the set $\Omega$ with socle $H$ is stated (and later proved)
in~\cite[Theorem 4.1A]{DM} (see also~\cite[Theorem 4.6A]{DM}).
The four possibilities when $H=S^k$ is nonabelian with $S$ simple are described
in Case (b) of the theorem. In Case (b)\,(i) $k=1$ and $G$ is almost simple.
In Case (b)\,(ii) $G$ is of diagonal type with $n=|S|^{k-1}$ and the point
stabiliser $H_\alpha$ is a diagonal subgroup of $S$.
In Case (b)\,(iv), $H$ acts regularly on $\Omega$, so $n=|S|^k$. Finally, in
Case (b)\,(iii), we have $k>1$, and $G$ is contained in a wreath product
$U \wr \Sym(k/d)$ with product action, where $U$ is a primitive permutation
group of almost simple or diagonal type with socle $S^d$ and degree $\ell$,
and $n = \ell^{k/d}$.
The group $U$ is almost simple if and only if $d=1$, in which case $n =\ell^k$,
and we shall say that $G$ is primitive of \emph{almost simple product type}.
Note also that $\ell = |S|^{d-1}$ when $U$ has diagonal type, so
$n \ge |S|^{k/2}$ in this case.

\begin{lemma}\label{lem:minprimdeg}
Suppose that $G \le \Sym(\Omega)$ has a unique minimal normal $N$, where
$N = S^k$ for some nonabelian simple group $S$, and suppose that $N^\Omega$
is transitive. Let $S$ be one of the simple factors of $N$, and let $T$ be
the almost simple group $N_G(S)/C_G(S)$. Then $|\Omega| \ge P(T)^k$.
Furthermore, there are embeddings $G \to T \wr \Sym(k) \to \Sym(P(T)^k)$ for
which the image of $N$ in $\Sym(P(T)^k)$ acts transitively.
\end{lemma}
\begin{proof} If $k=1$ then $G=T$ and the result is clear, so suppose that
$k>1$.  Since $N^\Omega$ is transitive, it acts non-trivially on any
block system preserved by $G$ on $\Omega$ so, by replacing $G$ by its action on any such block system, we may assume that $G^\Omega$ is primitive.
As we saw above, if $G$ is of almost simple product type then $n = \ell^k$,
where $\ell$ is the degree of a primitive permutation representation of an
almost simple group $U$ with socle $S$, and otherwise $n = |S|^t$ for some
$t \ge k/2$.
On the other hand, if $T = N_G(S)/C_G(S)$ as in the statement of the lemma,
then by~\cite[Section 3.2]{CH04} there is an embedding $G \to T \wr \Sym(k)$,
and the final assertion of the lemma follows by considering the product
action of $T \wr \Sym(k)$ using a (necessarily primitive) permutation
representation of $T$ of degree $P(T)$.

It is straightforward to check from~\cite[Table 4]{Guestetal} for simple groups
of Lie type, \cite{Atlas} for the sporadic groups,
and~\cite[Proposition 2.2]{CHU} for almost simple groups that, for any almost
simple group $U$ with socle $S$, we have $P(U)^2 < |S|$. So a faithful
permutation representation of $G$ of smallest possible degree such that $N$
acts transitively is primitive of almost simple product type. Furthermore,
given such a representation with $G$ embedding into $U \wr \Sym(k)$, the outer
automorphisms of $S$ that are induced by elements in $N_G(S)/S$ must be induced
by elements of $U$, and so we have $T \le U$, and the result follows.
\end{proof}

\begin{lemma}\label{lem:minwpquot}
Let $T$ be an almost simple group with socle $S$ and suppose that
$G \le T \wr \Sym(k)$, where $N = S^k \le G$.
Then $P(G/N) \le k|T/S| \le 2kP(T)/5$.
\end{lemma}
\begin{proof} Since $G/N \le (T/S) \wr \Sym(k)$, we have $P(G/N) \le k|T/S|$.
By~\cite[Lemma 4]{HT} we have $P(S) > 3|\Out(S)|$ except in the four
cases $(S,P(S),|\Out(S)|) = (A_6,6,4)$, $(\PSL(3,4),21,12)$, $(A_5,5,2)$, and
$(\PSL(2,8),9,3)$.  So, except in those four cases have
$k|T/S| \le k|\Out(S)| < kP(S)/3 \le kP(T)/3$, which proves the theeorem.
In fact we still have $k|T/S| \le 2kP(T)/5$ when \linebreak
$(S,P(S),|\Out(S)|)  = (A_5,5,2)$ or $(\PSL(2,8),9,3)$.
When $(S,P(S),|\Out(S)|) = (A_6,6,4)$ we have $\deg(T) \ge 10$ when $|T/S|=4$
and, when $(S,P(S),|\Out(S)|)  = (\PSL(3,4),21,12)$, we have
$P(T) \ge 42$ when $|T/S| =12$, so the result also holds in these cases.
\end{proof} 

\section{Proof of theorem}
We turn now to the proof of the theorem.  Assume that the group $G$ is acting
(faithfully) on the set $\Omega$ with $|\Omega|=n$.  We prove the theorem by
induction on $n$, and observe that, when proving the theorem for some
particular value of $n$, we can assume that the corollary holds for all $m < n$.
Let $G$ with normal subgroup $N$ be a counterexample of minimal degree.

\begin{lemma} $G$ is transitive on $\Omega$.
\end{lemma}
\begin{proof}
Since $N$ is a minimal normal subgroup of $G$, it acts either
trivially or faithfully on each orbit of $G$.
Suppose that $G$ is intransitive on $\Omega$, and let $\Delta$ be an
orbit of $G$ on $\Omega$ on which $N$ acts nontrivially, and hence faithfully.
Let $\Gamma = \Omega \setminus \Delta$.  By induction there are homomorphisms
$\rho_\Delta: G^\Delta \to \Sym(\Delta')$ with $\ker \rho_\Delta = N^\Delta$ and
$\rho_\Gamma: G^\Gamma \to \Sym(\Gamma)$ with $\ker \rho_\Gamma = N^\Gamma$,
where $\Delta'$ is a set with $|\Delta'|<|\Delta|$.
Define $\rho:G \to \Sym(\Delta' \cup \Gamma)$ by
$\rho(g) := \rho_1(g^\Delta)\rho_2(g^\Gamma)$.
We claim that $\ker \rho = N$, which contradicts the assumption that $G$ with
normal subgroup $N$ is a counterexample to the theorem.

Clearly $N \le \ker \rho$ so suppose that $g \in \ker \rho$.
Then $g^{\Delta} \in \ker \rho_1 = N^\Delta$, so there exists
$n_1 \in N$ with $g^\Delta = n_1^\Delta$, and then $gn_1^{-1}$
is in the kernel $K_1$ of the action of $G$ on $\Delta$.
Since $(gn_1^{-1})^\Gamma \in \ker \rho_2 = N^\Gamma$, there exists
$n_2 \in N$ with $gn_1^{-1}n_2^{-1}$ in the kernel $K_2$ of the action of
$G$ on $\Gamma$, and we also have
$gn_1^{-1}n_2^{-1} \in K_1 N = K_1 \times N$, because $N$ acts
faithfully on $\Delta$.
Now, by applying Lemma~\ref{lem:ntproj} to $K_1$, $N$, and
$L:=K_2 \cap K_1 N$, we find that any simple direct factor of $N$
onto which $n_2$ projects non-trivially lies in $K_2$, and so
$n_2 \in K_2$ and hence $gn_1^{-1} \in K_1 \cap K_2=1$,
so $g=n_1 \in N$.  This proves the claim and the lemma.
\end{proof}
 
\begin{lemma} $N$ is transitive on $\Omega$.
\end {lemma}
\begin{proof}
Suppose that $N$ is intransitive on $\Omega$. Then $G$ is imprimitive on
$\Omega$ and preserves the block system $\Gamma$, the set of orbits of $N$.
Let $\Delta \in \Gamma$.  Then by~\cite[Theorem 1.8]{Cameron} there is an
embedding $\alpha: G \to G_{\Delta}^\Delta \wr G^\Gamma$. Since
$N^\Delta \unlhd G_{\Delta}^\Delta$, there is a natural projection
$\pi: G_{\Delta}^\Delta \wr G^\Gamma \to
G_{\Delta}^\Delta/N^\Delta \wr G^\Gamma$.

We claim that the kernel of the composite map
$\pi\alpha:G \to G_{\Delta}^\Delta/N^\Delta \wr G^\Gamma$ is equal to $N$.
This kernel clearly contains $N$ so, for $g \in \ker \pi\alpha$, we need to
show that $g \in N$. Certainly $g$ is in the kernel $K$ of the induced action
of $G$ on $\Gamma$; that is, $g$ maps to an element of the base group of the
wreath product. Let $\Gamma = \{ \Delta=\Delta_1,\Delta_2,\ldots,\Delta_t\}$.
Then $g^{\Delta_i} \in N^{\Delta_i}$ for $1 \le i \le t$.
To prove the claim, we will successively replace $g$ by $gn$ for elements
$n \in N$ for $1 \le i \le t$ such that, after the $i$-th such replacement,
$g$ acts trivially on $\Delta_1, \cup  \cdots \cup \Delta_i$.
We can do this for $i=1$ because $g^{\Delta_1} \in N^{\Delta_1}$.

Suppose by induction that $g$ acts trivially on
$\Delta' := \Delta_1 \cup \cdots \cup \Delta_{i-1}$ for some $i$ with
$1 < i \le t$, and let $K_1$ and $K_2$ be the kernels of the actions of $K$
on $\Delta'$ and $\Delta_i$, respectively. So $g \in K_1$.

Since $N$ is a direct product of nonabelian simple groups that are permuted
under the action of conjugation by elements of $K$, we can write
$N = N_1 \times N_2$, where $N_1$ and $N_2$ are normal in $K$, and the simple
direct factors of $N$ in $N_1$ lie in $K_1$, but those of $N_2$ do not.
So $N_1 \le K_1$ and $K_1 \cap N_2 = 1$.
We have $g^{\Delta_i} = (n_1n_2)^{\Delta_i}$
for some $n_1 \in N_1$ and $n_2 \in N_2$, and then
$gn_1^{-1}n_2^{-1} \in K_2 \cap K_1N_2$.
Now, by applying Lemma~\ref{lem:ntproj} to the normal subgroups $K_1$, $N_2$,
and $L := K_2 \cap K_1N_2$, we find that all of the factors of $N_2$ onto which
$n_2$ projects nontrivially lie in $K_2$, so $n_2 \in K_2$ and hence
$gn_1^{-1} \in K_1 \cap K_2$,
which completes the induction and the proof of the claim.

Now, by the minimality of $G$ and $N$,we can apply Corollary~\ref{cor:main} to
$N^\Delta \unlhd G_{\Delta}^\Delta$, to find an embedding
$\rho: G_{\Delta}^\Delta/N^\Delta \to \Sym(\Delta')$
and $|\Delta'| \le 2|\Delta|/5$, which induces an embedding
$\bar{\rho} : G_{\Delta}^\Delta/N^\Delta \wr G^\Gamma \to \Sym(\Omega')$
with $|\Omega'| \le 2|\Omega|/5$.
Then the composite map $\bar{\rho}\pi\alpha:G \to \Sym(\Omega')$ has
kernel $N$, which proves the lemma.
\end{proof}

We have $N \cong S^k$ for some nonabelian simple group $S$ and $k \ge 1$,
where the simple direct factors are permuted transitively under the conjugation
action of $G$.  We have $\Aut(N) \cong \Aut(S) \wr \Sym(k)$, and
$\Out(N) \cong \Out(S) \wr \Sym(k)$, which embeds into $\Sym(k|\Out(S)|)$.
Let $C = C_G(N)$.

Suppose first that $C=1$. Then $N$ is the unique minimal normal subgroup
of $G$, and by Lemma~\ref{lem:minprimdeg} we have $n \ge |P(T)|^k$ with
$T = N_G(S)/C_G(S)$. Furthermore we have $P(G/N) \le 2kP(T)/5$ by
Lemma~\ref{lem:minwpquot}, and the result follows because $kP(T) \le P(T)^k$.
So we may assume that $|C|>1$.

For $\alpha \in \Omega$ let $H:=G_\alpha$ be a point stabiliser,
so $n = |G:H|$ and, since $N^\Omega$ is transitive, we have $G=NH$.
We have $C_{\Sym(\Omega)}(N) \cong N_N(H \cap N)/H \cap N$
by~\cite[Theorem 4.2A\,(i)]{DM}. Since $G=NH$, the simple direct factors of $N$
are all conjugate under the conjugation action of $H$, so
the projection of $H \cap N$ onto each of the simple direct factors of $N$
must be the same. If this projection is equal to $S$, then
$N_N(H \cap N) = H \cap N$ by Lemma~\ref{lem:normsd}, so $C=1$ and
the result follows as above. On the other hand, if $H\cap N=1$, then
$N^\Omega$ is regular so $n=|N|=|S|^k$ and, by~\cite[Theorem 4.2A\,(iii)]{DM}
we have $G/N \le N_{\Sym(\Omega)}(N)/N \cong \Aut(S) \wr \Sym(k)$, so
$P(G/N)\le kP(\Aut(S)) < 2|S|/5 \le 2n/5$.

So we can assume that the projection of $H \cap N$ onto the simple direct
factors of $N$ is a nontrivial proper subgroup $R$ of $S$ and hence
$H \cap N \le R^k$.  Let $a := |R^k:H \cap N|$, $Q := N_S(R)$, and $b:= |Q:R|$.
Then $N_N(H \cap N) \le Q^k$, so $|C| \le |N_N(H \cap N)/H \cap N| \le ab^k$.

Since $H$ acting by conjugation permutes the factors of $S^k$ and $H$
normalises $H \cap N$, we have $H \le N_G(R^k)$ and hence $H \le N_G(Q^k)$.
So $H \le N_G(Q^kC)$, and we can choose a maximal subgroup $M$ of $G$
containing $HQ^kC$. Let $\Gamma$ be the set of blocks of imprimitivity
of $G$ containing the block $\alpha^M$. Then, since $N/C$ is the unique
minimal normal subgroup of $G/C$, the kernel of the action of $G$ on
$\Gamma$ is equal to $C$ and, since the stabiliser of a block contains $Q^k$,
we see from Lemma~\ref{lem:minprimdeg} and the preceding discussion
of the cases of the O'Nan-Scott Theorem that the induced action $G^\Gamma$ of
$G$ on $\Gamma$ must be primitive of almost simple product type and contained
in $T \wr \Sym(k)$ with product action, where $T \le \Aut(S)$ and
$|\Gamma| = \deg(T)^k$.

Now, by Lemma~\ref{lem:prodactorb},
in this action of $T \wr \Sym(k)$, the subgroup $\Sym(k)$ has orbits of length
$k$ on which it acts naturally. Let $D^\Gamma$ (with $C \le D \le G)$ be the
intersection of $G^\Gamma$ with this subgroup. Then $D$ acts faithfully
on a union of $k$ blocks in $\Gamma$, a set of size $k|M:H|$, and $D$
is isomorphic to a subgroup of index at most $|T/S|^k$ in $G/N$. 
So, by Lemma~\ref{lem:subembed}, there is a faithful action of $G/N$ on a
set of size at most $k|T/S|^k|M:H|$.

Since $n = |\Gamma||M:H|$ with $|\Gamma| = \deg(T)^k \ge P(T)^k$
and $|T/S| \le 2P(T)/5$ by Lemma~\ref{lem:minwpquot}, we have
\begin{equation*}
\begin{aligned}
P(G/N) & \le k|T/S|^k|M:H|  \le k(2P(T)/5)^k|M:H|\\
&\le  k(2/5)^k|\Gamma||M:H| = k(2/5)^kn,
\end{aligned}
\end{equation*}
and the result follows because $k(2/5)^k \le 2/5$ for all $k \ge 1$.

\end{document}